\documentclass[preprint,12pt]{article}

\usepackage{amssymb}

\usepackage{amsmath}
\usepackage{amsfonts}

\newtheorem{theorem}{Theorem}
\newtheorem{proposition}[theorem]{Proposition}
\newtheorem{lemma}[theorem]{Lemma}
\newtheorem{corollary}[theorem]{Corollary}
\newtheorem{remark}[theorem]{Remark}
\newtheorem{remarks}[theorem]{Remarks}
\newenvironment{proof}[1][Proof]{\noindent\textbf{#1.} }{\ \rule{0.5em}{0.5em}}

\usepackage{setspace}
\begin{document}




\title{Analyzing a Seneta's conjecture by using the Williamson transform}


\author{Edward Omey and Meitner Cadena}




\maketitle

\begin{abstract}
Considering slowly varying functions (SVF), 
Seneta in 2019 conjectured the following implication, for $\alpha\geq1$,

$$
\int_0^x y^{\alpha-1}(1-F(y))dy\textrm{\ is SVF}\ \Longrightarrow\ \int_{[0,x]}y^{\alpha}dF(y)\textrm{\ is SVF, as $x\to\infty$,}
$$

where $F(x)$ is a cumulative distribution function on $[0,\infty)$.
By applying the Williamson transform, an extension of this conjecture is proved.
Complementary results related to this transform and particular cases of this extended conjecture are discussed.
\end{abstract}



\begin{quote}



\end{quote}



A function $f(x)$ is slowly varying (SVF) if, for any $t>0$, $f(tx)\big/f(x)\to1$ as $x\to\infty$.
If $F(x)$ is a distribution function and $\overline{F}(x)=1-F(x)$ is its tail,
recently Seneta \cite{Seneta} 
conjectured the following.
Given $\alpha\geq1$,

\begin{equation}\label{conjectureSeneta}
\int_0^x y^{\alpha-1}\overline{F}(y)dy\ \textrm{is SVF}\quad\Longrightarrow\quad
\int_{[0,x]} y^{\alpha}dF(y)\ \textrm{is SVF}\textrm{, as $x\to\infty$.}
\end{equation}

Nowadays, 
Kevei in \cite{Kevei} 
presented an extension of such a conjecture and its proof.
Coincidentally, we have been working on such a subject as a part of other researches related to results shown by 
\cite{Seneta},
see 
\cite{OmeyCa} and 
\cite{OmeyCb}.
Unlike Kevei's proofs, ours are based on mainly the Williamson transform.
Also, the application of this transform to $F$ allows the formulation of another conjecture more.

In the following section, we present our main results.
Previously, we start introducing notation and preliminary results.

\section{Main results}

\subsection{Some notation and transforms}

In what follows, $F(x)$ denotes a distribution function (d.f.) defined on $[0,\infty)$ with $F(0)=0$, and $\overline{F}(x)=1-F(x)$ denotes its tail.
The $\alpha$--th moment of $F(x)$ is denoted by $m(\alpha)$.
It is said that $f(x)\sim g(x)$ if $f(x)\big/ g(x)\to 1$ as $x\to\infty$.
The case $f(x)\sim 0$ will be understood as $f(x)\to 0$ as $x\to\infty$.
The class of regularly varying functions with index $\alpha$, denoted by $RV_{\alpha}$, consists of functions $f(x)$ satisfying, for any $t>0$, $f(tx)\sim t^{\alpha}f(x)$.
If the index $\alpha=0$, $f(x)$ is said slowly varying.
If $L\in RV_{0}$, then the de Haan class, denoted by $\Pi_{\beta}(L)$, consists of functions $f(x)$ satisfying, for any $t>1$, $(f(tx)-f(x))\big/L(x)\sim\beta\log t$.


Let $\alpha >0$. 

We define the following
transformations:%
$$
H_{\alpha }(x)=\int_{0}^{x}y^{\alpha }dF(y)\textrm{,}  
$$
and%
$$
W_{\alpha }(x)=\int_{0}^{x}y^{\alpha -1}\overline{F}(y)dy\textrm{.}  
$$
Note that Lebesgue's
theorem on dominated convergence shows that $\lim_{x\rightarrow \infty
}x^{-\alpha }H_{\alpha }(x)=0$. 
Hence, we have%
\begin{equation}
H_{\alpha }(x)=\alpha \int_{0}^{x}y^{\alpha -1}\overline{F}(y)dy-x^{\alpha }%
\overline{F}(x)\textrm{.}  \label{3}
\end{equation}%

Then, we 
get the following result.

\begin{proposition}\label{prop1}
We have, for $x\geq 0$,
\begin{equation}
\overline{F}(x)=\alpha \int_{x}^{\infty }z^{-\alpha -1}H_{\alpha
}(z)dz-x^{-\alpha }H_{\alpha }(x)\textrm{.}  \label{5}
\end{equation}
\end{proposition}

\begin{proof}
Let us consider the integral $\int_{x}^{a}z^{-\alpha -1}H_{\alpha }(z)dz$ for $%
a\geq x\geq 0$. We have, by using (\ref{3}), %
\begin{eqnarray*}
\lefteqn{\int_{x}^{a}z^{-\alpha -1}H_{\alpha }(z)dz} \\
 & = & \int_{x}^{a}z^{-\alpha
-1}\left(H_{\alpha }(x)+\int_{x}^{z}y^{\alpha }dF(y)\right)dz \\
&=&\frac{1}{\alpha }(x^{-\alpha }-a^{-\alpha })H_{\alpha
}(x)+\int_{x}^{a}\int_{y}^{a}z^{-\alpha -1}y^{\alpha }dzdF(y) \\
&=&\frac{1}{\alpha }(x^{-\alpha }-a^{-\alpha })H_{\alpha }(x)+\frac{1}{%
\alpha }\int_{x}^{a}(y^{-\alpha }-a^{-\alpha })y^{\alpha }dF(y) \\
&=&\frac{1}{\alpha }(x^{-\alpha }-a^{-\alpha })H_{\alpha }(x)+\frac{1}{%
\alpha }(F(a)-F(x))-\frac{1}{\alpha }a^{-\alpha }(H_{\alpha }(a)-H_{\alpha
}(x)) \\
&=&\frac{1}{\alpha }x^{-\alpha }H_{\alpha }(x)+\frac{1}{\alpha }(F(a)-F(x))-%
\frac{1}{\alpha }a^{-\alpha }H_{\alpha }(a)\textrm{.}
\end{eqnarray*}%
Taking limits as $a\rightarrow \infty $, we obtain that $a^{-\alpha
}H_{\alpha }(a)\rightarrow 0$ and hence also that%
$$
\int_{x}^{\infty }z^{-\alpha -1}H_{\alpha }(z)dz=\frac{1}{\alpha }x^{-\alpha
}H_{\alpha }(x)+\frac{1}{\alpha }\overline{F}(x)\textrm{.} 
$$
This proves the result.
\end{proof}


Now, we recall
the Williamson transform which is
defined by, see 
\cite{Williamson},
%
$$
G_{\alpha}(x)=\int_{0}^{x}\left(1-\left(\frac{t}{x}\right)^{\alpha }\right)dF(t)\textrm{.}
$$
Note that 
$$
G_{\alpha}(x)=\int_{0}^{\infty }P\left(Z>\frac{t}{x}\right)dF(t)=EF(Zx)=P\left(\frac{X}{Z}\leq x\right)%
\textrm{,}
$$
where $Z$ is independent of $X$ and $P(Z\leq x)=x^{\alpha },0\leq x\leq 1$.
Among others, this shows that $G_{\alpha}(x)$ is also a d.f. (with $G_{\alpha}(0)=0$).
Also, we have the following result.

\begin{proposition}\label{Newprop1}
We have, for $x\geq 0$,

\begin{itemize}
\item[(i)] $G_{\alpha}(x)=\alpha x^{-\alpha
}\int_{0}^{x}t^{\alpha -1}F(t)dt$.

\item[(ii)] $F(x)=G_{\alpha}(x)+\frac{x}{\alpha }G_{\alpha}^{\prime }(x)$.

\end{itemize}
\end{proposition}

\begin{proof}
(i) We have, by using partial integration,

$$
G_{\alpha}(x)=F(x)-x^{-\alpha }\int_{0}^{x}t^{\alpha }dF(t)=\alpha x^{-\alpha
}\int_{0}^{x}t^{\alpha -1}F(t)dt\textrm{.}
$$

(ii)
Writing the previous result as $x^{\alpha
}G_{\alpha}(x)=\alpha \int_{0}^{x}t^{\alpha -1}F(t)dt$,
deriving this relation, and dividing it by $x^{\alpha -1}\big/\alpha$ give the result claimed.
%
\end{proof}


The following result provides relations among $\overline{F}(x)$, $\overline{G}_{\alpha}(x)$, $H_{\alpha }(x)$ and $W_{\alpha }(x)$.

\begin{proposition}\label{propg}
We have, for $x\geq 0$,

\begin{itemize}
\item[(i)] $\overline{F}(x)=\overline{G}_{\alpha}(x)-x^{-\alpha }H_{\alpha }(x)$.

\item[(ii)] $\overline{G}_{\alpha}(x)=\alpha x^{-\alpha }W_{\alpha }(x)$.

\end{itemize}
\end{proposition}

\begin{proof}
(i) Using $H_{\alpha }(x)$ and 
Proposition \ref{Newprop1} (i), we have that $G_{\alpha}(x)=F(x)-x^{-\alpha }H_{\alpha
}(x)$ and $\overline{G}_{\alpha}(x)=\overline{F}(x)+x^{-\alpha }H_{\alpha }(x)$.

(ii) Using the tail of $G_{\alpha}(x)$, $W_{\alpha }(x)$, and 
Proposition \ref{Newprop1} (i), we find%
$$
\overline{G}_{\alpha}(x)=\alpha x^{-\alpha }\int_{0}^{x}t^{\alpha -1}\overline{F}%
(t)dt=\alpha x^{-\alpha }W_{\alpha }(x)\textrm{.}
$$
\end{proof}

%
%
%

\subsection{Main relations}

It follows our first main result.

\begin{theorem}\label{th1}
Let $\alpha>0$ and $0\leq\theta \leq \alpha $.
We have:

\begin{itemize}
	\item[(i)]
If $0<\theta <\alpha $, then the following statements are equivalent:
\begin{itemize}
	\item[(a)] $\overline{F}(x)\in RV_{\theta-\alpha}$;
	\item[(b)] $\overline{G}_{\alpha}(x)\in RV_{\theta-\alpha}$;
	\item[(c)] $W_{\alpha}(x)\in RV_{\theta}$;
	\item[(d)] $\overline{F}(x)\big/\overline{G}_{\alpha}(x)\sim\theta\big/\alpha$;
	\item[(e)] $H_{\alpha}(x)\in RV_{\theta}$;
	\item[(f)] $x^{\alpha}\overline{F}(x)\big/H_{\alpha}(x)\sim\theta\big/(\alpha-\theta)$;
	\item[(g)] $x^{\alpha }\overline{G}_{\alpha}(x)\big/H_{\alpha}(x)\sim \alpha \big/(\alpha -\theta)$; and,
	\item[(h)] $x^{\alpha }\overline{F}(x)\big/W_{\alpha}(x)\sim \theta$.
\end{itemize}

	\item[(ii)]
If $\theta=0$, then the statements (b), (c), (d), (e), (f), (g) and (h) are equivalent.

	\item[(iii)]
If $\theta=\alpha$, then the statements (a), (b) and (c) are equivalent, and (b) $\Longrightarrow$ (d).

\end{itemize}
\end{theorem}

\begin{remarks}~

\begin{itemize}
	\item[1)]
	Taking $\theta=0$, Theorem \ref{th1}(e), -(f) and -(h) correspond to (17), (18) and (19) in Theorem 3 by 
	\cite{Seneta}
	, respectively.

	\item[2)] Theorem \ref{th1}(a), -(e) and -(f) are proved in VIII.9 by 
	\cite{Feller}, see e.g. Theorem 8.1.2 in 
	\cite{Bingham}.

\item[3)]
The conjecture indicated in
(\ref{conjectureSeneta}) is contained and extended in the implication, for $0\leq\theta<\alpha$, Theorem \ref{th1}(c) $\Longrightarrow$ Theorem \ref{th1}(e).

	\item[4)] Theorem \ref{th1}(a), -(c), -(e), -(h) and a combination of -(f) and -(h) correspond to (1.6), (1.4), (1.5), (1.7) and (1.8) in Theorem 1.1 by 
	\cite{Kevei}, respectively.

\end{itemize}
\end{remarks}

We continue with results when some moments are finite.

\begin{theorem}\label{th2}
Let $\alpha>0$ and $\theta>\alpha $.
	Assume $m(\alpha)<\infty$. 
	Set $\overline{W}_{\alpha}(x)=W_{\alpha}(\infty)-W_{\alpha}(x)$.
	The following statements are equivalent:

	\begin{itemize}
		\item[(a)] $\overline{F}(x)\in RV_{-\theta}$;
		\item[(b)] $\overline{W}_{\alpha}(x)\in RV_{\alpha-\theta}$;
		\item[(c)] $x^{-\alpha}m(\alpha)-\overline{G}_{\alpha}(x)\in RV_{-\theta}$;
		\item[(d)] $\overline{W}_{\alpha}(x)\big/(x^{\alpha}\overline{F}(x))\sim 1\big/(\theta-\alpha)$;
		\item[(e)] $(x^{-\alpha}m(\alpha)-\overline{G}_{\alpha}(x))\big/\overline{F}(x)\sim \alpha\big/(\theta-\alpha)$.
		\item[(f)] $m(\alpha )-H_{\alpha }(x)\in RV_{\alpha -\theta }$; and,
		\item[(g)] $(m(\alpha )-H_{\alpha }(x))\big/(x^{\alpha}\overline{F}(x))\sim\theta\big/({\theta -\alpha })$.
	\end{itemize}
\end{theorem}

\begin{remark}
	Theorem \ref{th2} (a), -(f) and -(g) are proved in VIII.9 by 
	\cite{Feller}, see e.g. Theorem 8.1.2 in 
\cite{Bingham}.
\end{remark}

Next, results involving the de Haan class follow.

\begin{theorem}\label{th3}
Let $\alpha>0$.
Let $L(x)\in RV_{0}$.
Let $\beta,\lambda\geq0$ so that $\beta =\alpha \lambda $.
	The following statements are equivalent:
\begin{itemize}
	\item[(a)] $x^{\alpha }\overline{G}_{\alpha}(x)\in \Pi _{\beta
}(L)$;
	\item[(b)] $x^{\alpha }\overline{F}(x)\big/L(x)\sim \lambda $; and,
	\item[(c)] $H_{\alpha }(x)\in \Pi _{\beta }(L)$
\end{itemize}
\end{theorem}

\begin{remark}
The relation (b) $\Longleftrightarrow$ (c) was identified in Theorem 1.1 by 
\cite{Kevei}.
\end{remark}

To prove the previous theorems, we present several results which 
are organized in subsections by considering relations among $\overline{F}$, $H_{\alpha }$, $W_{\alpha }$ and $\overline{G}_{\alpha }$.
They are based mainly on the following well-known result, see 
\cite{Karamata} 
and e.g. Theorem 1.2.1 by 
\cite{dehaan}:

\begin{proposition}\label{karamatarv}
Suppose $U:\mathbb{R}^+\to\mathbb{R}^+$ is Lebesgue-summable on finite intervals. Then
\begin{eqnarray}
U(x)\in RV_{\alpha}\textrm{, }\alpha>-1 & \textrm{iff} & xU(x)\sim (\alpha+1)\int_0^x U(t)dt\textrm{;} \label{k1} \\
 & & \nonumber \\
U(x)\in RV_{\alpha}\textrm{, }\alpha<-1 & \textrm{iff} & xU(x)\sim -(\alpha+1)\int_x^{\infty} U(t)dt\textrm{.} \label{k2}
\end{eqnarray}
\end{proposition}

The following proposition is also used to prove some of those results.

\begin{proposition}\label{proposition20210716}
Let $\alpha>0$.
Assume that there exist $A(x)$, $B(x)$, and $C(x)$ satisfying, for any $z>1$,
$$
\frac{W(zx)-W(x)}{A(x)}\to B(z)\textrm{\quad and \quad}\frac{W(zx)-W(x)}{A(zx)}\to C(z)\textrm{,}
$$
such that, for some $\xi$, $B(z)\big/(z^{\alpha }-1)\to\xi\big/\alpha$ and $C(z)\big/(1-z^{-\alpha })\to\xi\big/\alpha$.
Then, we have
$$
\frac{\overline{F}(x)x^{\alpha }}{A(x)}\to\xi\textrm{.}
$$
\end{proposition}

\begin{proof}
For $z>1$, we have $W_{\alpha }(zx)-W_{\alpha }(x)=\int_{x}^{zx}y^{\alpha -1}%
\overline{F}(y)dy$. Since $\overline{F}(x)$ is nonincreasing, we find that%
$$
\frac{1}{\alpha }\overline{F}(xz)t^{\alpha }(z^{\alpha }-1)\leq W_{\alpha
}(zx)-W_{\alpha }(x)\leq \frac{1}{\alpha }\overline{F}(x)x^{\alpha
}(z^{\alpha }-1)\textrm{,} 
$$
and, then, also that%
\begin{eqnarray*}
\frac{W_{\alpha }(zx)-W_{\alpha }(x)}{A(x)} &\leq &\frac{x^{\alpha
}\overline{F}(x)}{A(x)}\frac{1}{\alpha }(z^{\alpha }-1)\textrm{,} \\
\frac{(xz)^{\alpha }\overline{F}(xz)}{A(xz)}\frac{1}{\alpha }%
z^{-\alpha }(z^{\alpha }-1) &\leq &\frac{W_{\alpha }(zx)-W_{\alpha }(x)}{%
A(zx)}\textrm{.}
\end{eqnarray*}%
Taking limits, we obtain that%
\begin{eqnarray*}
\alpha \frac{B(z)}{z^{\alpha }-1} &\leq &\liminf_{x\to\infty} \frac{\overline{F%
}(x)x^{\alpha }}{A(x)}\textrm{,} \\
\limsup_{x\to\infty} \frac{(xz)^{\alpha }\overline{F}(xz)}{A(xz)} &\leq
&\alpha \frac{C(z)}{1-z^{-\alpha }}\textrm{,}
\end{eqnarray*}%
or 
$$
\alpha \frac{B(z)}{z^{\alpha }-1}\leq \lim \binom{\sup }{\inf }%
 \frac{\overline{F}(x)x^{\alpha }}{W_{\alpha }(x)}\leq \alpha \frac{C(z)%
}{1-z^{-\alpha }}\textrm{.} 
$$
Now, let $z\downarrow 1$ to find, by hypothesis, that $x^{\alpha }\overline{F}(x)\big/A(x)\sim \xi $.
\end{proof}

\begin{remark}\label{remark20210716}
If $W_{\alpha }(\infty)<\infty$, we have
$\overline{W}_{\alpha }(x)-\overline{W}_{\alpha }(xz)=W_{\alpha }(xz)-W_{\alpha }(x)$.
Hence, we can again get Proposition \ref{proposition20210716}, but taking $\overline{W}_{\alpha }(x)$ instead of ${W}_{\alpha }(x)$.
\end{remark}

\subsubsection{Relations among $\overline{F}(x)$, $W_{\protect\alpha }(x)$
and $\overline{G}_{\alpha}(x)$}


\begin{lemma}\label{lemma1}
Let $\alpha>0$.
We have:

\begin{itemize}
	\item[(i)] Suppose that $0<\theta \leq \alpha $. We have $\overline{F}(x)\in
RV_{\theta -\alpha }(x)$ iff $\overline{G}_{\alpha}(x)\in RV_{\theta -\alpha }$ iff $%
W_{\alpha }(x)\in RV_{\theta }$. Each of the statements implies that $%
\overline{F}(x)\big/\overline{G}_{\alpha}(x)\sim \theta \big/\alpha $.

	\item[(ii)] $\overline{G}_{\alpha}(x)\in RV_{-\alpha }$ iff $
W_{\alpha }(x)\in RV_{0}$. Each of the statements implies that $\overline{F}(x)\big/\overline{G}%
_{\alpha}(x)\sim 0$. 

	\item[(iii)] If $\overline{F}(x)\big/\overline{G}_{\alpha}(x)\sim \lambda$, $0\leq
\lambda < 1$, then 
$\overline{G}_{\alpha}\in RV_{\alpha (\lambda -1)}$.

	\item[(iv)] Let $L(x)\in RV_{0}$. Then $x^{\alpha }\overline{G}_{\alpha}(x)\in \Pi _{\beta
}(L)$ iff $x^{\alpha }\overline{F}(x)\big/L(x)\sim \lambda $ with $\beta =\alpha \lambda $%
.

\end{itemize}
\end{lemma}

\begin{proof}
Proof of (i). Consider $\overline{G}_{\alpha}(x)=\alpha x^{-\alpha
}\int_{0}^{x}y^{\alpha -1}\overline{F}(y)dy$ because of Proposition \ref{propg}~(ii).
This implies that $\overline{G}_{\alpha}(x)\in RV_{\theta -\alpha }$ iff $%
W_{\alpha }(x)\in RV_{\theta }$.
Next, assume that $\overline{F}(x)\in RV_{\theta -\alpha }$, $0<\theta \leq
\alpha $. Consider again $\overline{G}_{\alpha}(x)=\alpha x^{-\alpha
}\int_{0}^{x}y^{\alpha -1}\overline{F}(y)dy$ and note that $\alpha -1+\theta -\alpha
=\theta -1>-1$. We have, by applying (\ref{k1}), $\overline{G}_{\alpha}(x)\sim \alpha \overline{F}%
(x)\big/\theta $. Hence, $\overline{G}_{\alpha}(x)\in RV_{\theta -\alpha }$.
Also, we have that $\overline{F}(x)\in RV_{\theta -\alpha }$ implies that
$\overline{G}_{\alpha}(x)\big/\overline{F}%
(x)\sim \alpha \big/\theta $.

Now, assume that $W_{\alpha }(x)\in RV_{\theta }$ and $0<\theta \leq
\alpha $.
By applying Proposition \ref{proposition20210716} with $A(x)=W_{\alpha }(x)$, $B(x)=x^{\theta}-1$, $C(x)=1-x^{-\theta}$, and $\xi=\theta$, we
find that $x^{\alpha }\overline{F}(x)\big/W_{\alpha
}(x)\sim \theta $ and, hence, $\overline{F}\in RV_{\theta -\alpha }$.

Proof of (ii). 
Consider again $\overline{G}_{\alpha}(x)=\alpha x^{-\alpha
}\int_{0}^{x}y^{\alpha -1}\overline{F}(y)dy=\alpha x^{-\alpha
}W_{\alpha}(x)$.
Then, it follows that $\overline{G}_{\alpha}(x)\in RV_{-\alpha }$ iff $
W_{\alpha }(x)\in RV_{0}$.
Assume that $%
W_{\alpha }(x)\in RV_{0}$.
By applying Proposition \ref{proposition20210716} with $A(x)=W_{\alpha }(x)$, $B(x)=0$, $C(x)=0$, and $\xi=0$, we get
$x^{\alpha }\overline{F}(x)\big/W_{\alpha
}(x)\sim 0$ and, thus, $\overline{F}(x)\big/\overline{G}_{\alpha}(x)\sim 0$.

Proof of (iii). Assume $\overline{F}(x)\big/\overline{G}_{\alpha}(x)\sim \lambda$, $0\leq
\lambda < 1$. By Proposition \ref{Newprop1}, 
we have that $\overline{F}(x)=\overline{G}_{\alpha}(x)-xG_{\alpha}^{\prime
}(x)\big/\alpha $. 
We get
\begin{equation*}
\frac{xG_{\alpha}^{\prime }(x)}{\overline{G}_{\alpha}(x)}=\alpha \frac{\overline{G}%
_{\alpha}(x)-\overline{F}(x)}{\overline{G}_{\alpha}(x)}\sim \alpha (1-\lambda) \textrm{.}
\end{equation*}%
Hence, by applying L'Hopital rule, we have
$$
\frac{\overline{G}_{\alpha}(x)}{\int_x^{\infty} y^{-1}\overline{G}_{\alpha}(y)dy}\sim \alpha (1-\lambda
)\textrm{.} 
$$
Then, noticing $-\alpha (1-\lambda
)-1<-1$, by applying (\ref{k2}), we have $x^{-1}\overline{G}_{\alpha}(x)\in RV_{-\alpha (1-\lambda
)-1}$.
This result is equivalent to $\overline{G}_{\alpha}(x)\in RV_{\alpha (\lambda -1)}$.

Proof of (iv). First, assume that $x^{\alpha }\overline{F}(x)\big/L(x)\sim \lambda $%
. Using $W_{\alpha }(x)$, in the proof of Proposition \ref{proposition20210716}
we showed that for $z>1$, we have%
\begin{equation*}
W_{\alpha }(zx)-W_{\alpha }(x)=\int_{x}^{zx}y^{\alpha -1}\overline{F}(y)dy%
\textrm{.} 
\end{equation*}%
Then, we find that%
\begin{eqnarray*}
\frac{W_{\alpha }(zx)-W_{\alpha }(x)}{L(x)} &=&\int_{x}^{zx}y^{-1}\frac{%
y^{\alpha }\overline{F}(y)}{L(x)}dy \\
&=&\int_{1}^{z}y^{-1}\frac{(xy)^{\alpha }\overline{F}(xy)}{L(xy)}\frac{L(xy)%
}{L(x)}dy \\
&\sim &\lambda \log z\textrm{.}
\end{eqnarray*}%
It follows that $W_{\alpha }(x)\in \Pi _{\lambda }(L)$.
Because of
$$
(zx)^{\alpha}\overline{G}_{\alpha }(zx)-x^{\alpha}\overline{G}_{\alpha }(x)=
\alpha\left(W_{\alpha }(zx)- W_{\alpha }(x)\right)\textrm{,}
$$
we also have $x^{\alpha}\overline{G}%
_{\alpha}(x)\in \Pi _{\beta }(L)$ with $\beta =\alpha \lambda $.


For the converse, noting that $L(zx)\big/ L(x)\sim1$ for any $z>0$,
by applying Proposition \ref{proposition20210716} with $A(x)=L(x)$, $B(x)=C(x)=\lambda \log x$, and $\xi=\lambda$, we
find that $\overline{F}(x)x^{\alpha
}\big/L(x)\sim \lambda $.
\end{proof}


\begin{remark}\label{remarks1}
%
If $W_{\alpha }(\infty )<\infty $, we have%
\begin{equation*}
\overline{W}_{\alpha }(x)\equiv W_{\alpha }(\infty )-W_{\alpha
}(x)=\int_{x}^{\infty }y^{\alpha -1}\overline{F}(y)dy\textrm{.} 
\end{equation*}%
As to $\overline{G}_{\alpha}(x)$, we have%
\begin{equation*}
x^{-\alpha }m(\alpha )-\overline{G}_{\alpha}(x)=\alpha x^{-\alpha }\overline{W}%
_{\alpha }(x)\textrm{.} 
\end{equation*}%

\end{remark}

\begin{lemma}\label{lemma1bis}
Let $\alpha>0$.
Suppose that $W_{\alpha }(\infty )<\infty $ and $\theta >\alpha $. 
The following are equivalent:
\begin{itemize}
	\item[(i)] $\overline{F}(x)\in RV_{-\theta }$;
	\item[(ii)] $%
\overline{W}_{\alpha }(x)\in RV_{\alpha -\theta }$;
	\item[(iii)] $x^{-\alpha }m(\alpha )-%
\overline{G}_{\alpha}(x)\in RV_{-\theta }$;
	\item[(iv)] $\overline{W}_{\alpha }(x)\big/(x^{\alpha }\overline{F}(x)) \sim %
1\big/(\theta -\alpha )$; and,
	\item[(v)] $(x^{-\alpha }m(\alpha )-\overline{G}_{\alpha}(x))\big/\overline{F}(x)
\sim \alpha \big/(\theta -\alpha )$.
\end{itemize}
\end{lemma}

\begin{proof}
Assume that $\theta >\alpha $.

Proof of (i) $\Longrightarrow$ (ii).
First, assume that $\overline{F}(x)\in RV_{-\theta }$.
Note that $\overline{W}_{\alpha}(\infty)<\infty$.
Hence, because of $\overline{W}_{\alpha}(x)=\int_x^{\infty}y^{\alpha-1}\overline{F}(y)dy$, $x^{\alpha-1}\overline{F}(x)\in RV_{\alpha-1-\theta }$ and $\alpha-1-\theta<-1$,
we have, by applying (\ref{k2}),
$$
\frac{x^{\alpha}\overline{F}(x)}{\overline{W}_{\alpha}(x)}=\frac{x^{\alpha}\overline{F}(x)}{\int_x^{\infty}y^{\alpha-1}\overline{F}(y)dy}
\sim\theta-\alpha\textrm{.}
$$
This relation also implies that (i) $\Longrightarrow$ (iv).

Proof of (iv) $\Longrightarrow$ (i).
Now, assume that $\overline{W}_{\alpha }(x)\big/(x^{\alpha }\overline{F}(x)) \sim %
1\big/(\theta -\alpha )$.
Then, noting that $-(\theta -\alpha+1)=-\theta +\alpha-1<-1$ and
$$
\frac{x^{\alpha}\overline{F}(x)}{\overline{W}_{\alpha}(x)}=\frac{x^{\alpha}\overline{F}(x)}{\int_x^{\infty}y^{\alpha-1}\overline{F}(y)dy}\textrm{,}
$$
by applying  (\ref{k2}), give $x^{\alpha-1}\overline{F}(x)\in RV_{\alpha-1-\theta }$.
Hence, we have $\overline{F}(x)\in RV_{-\theta }$.

Proof of (ii) $\Longrightarrow$ (i).
Now, assume that $%
\overline{W}_{\alpha }(x)\in RV_{\alpha -\theta }$.
By applying Remark \ref{remark20210716} with $A(x)=\overline{W}_{\alpha }(x)$, $B(x)=x^{\alpha -\theta}-1$, $C(x)=1-x^{-(\alpha -\theta)}$, and $\xi=\theta-\alpha$, we
get
$x^{\alpha }\overline{F}(x)\big/\overline{W}_{\alpha }(x)\sim \theta-\alpha $ and, hence, $\overline{F}\in RV_{-\theta}$.

Proof of (i) $\Longleftrightarrow$ (iii).
Previously, we proved that (i) $\Longleftrightarrow$ (ii), i.e. $\overline{F}(x)\in RV_{-\theta }$ $\Longleftrightarrow$ $%
\overline{W}_{\alpha }(x)\in RV_{\alpha -\theta }$.
From 
Remark \ref{remarks1}, we have $x^{-\alpha }m(\alpha )-%
\overline{G}_{\alpha}(x)=\alpha x^{-\alpha}\overline{W}_{\alpha }(x)$.
Hence, we have clearly that $\overline{F}(x)\in RV_{-\theta }$ $\Longleftrightarrow$ $x^{-\alpha }m(\alpha )-%
\overline{G}_{\alpha}(x)\in RV_{-\theta }$.

Proof of (i) $\Longleftrightarrow$ (iv) $\Longleftrightarrow$ (v).
Again, considering
$x^{-\alpha }m(\alpha )-%
\overline{G}_{\alpha}(x)=\alpha x^{-\alpha}\overline{W}_{\alpha }(x)$,
we have
$$
\frac{x^{-\alpha }m(\alpha )-%
\overline{G}_{\alpha}(x)}{\overline{F}(x)}=
\alpha \frac{x^{-\alpha}\overline{W}_{\alpha }(x)}{\overline{F}(x)}
=\alpha \frac{\int_x^{\infty}y^{\alpha-1}\overline{F}(y)dy}{x^{\alpha}\overline{F}(x)}\textrm{.}
$$
Then, noticing that $-\theta+\alpha-1<-1$, we have that, by applying (\ref{k2}),
$$
\frac{x^{-\alpha }m(\alpha )-%
\overline{G}_{\alpha}(x)}{\overline{F}(x)}=
\alpha \frac{x^{-\alpha}\overline{W}_{\alpha }(x)}{\overline{F}(x)}=
\alpha \frac{\int_x^{\infty}y^{\alpha-1}\overline{F}(y)dy}{x^{\alpha}\overline{F}(x)}\sim\frac{\alpha}{\theta-\alpha}\textrm{,}
$$
is equivalent to $x^{\alpha-1}\overline{F}(x)\in RV_{-\theta+\alpha-1}$, i.e. $\overline{F}(x)\in RV_{-\theta}$.
The claim then follows.
\end{proof}


\subsubsection{Relations between $\overline{F}(x)$ and $H_{\protect\alpha %
}(x)$}


\begin{lemma}\label{lemma2}
Let $\alpha>0$.
We have:

\begin{itemize}
	\item[(i)] If $0<\theta <\alpha$, then $H_{\alpha }(x)\in RV_{\theta }$ iff $\overline{F}(x)\in
RV_{\theta -\alpha }$.

	\item[(ii)] If $0\leq \theta <\alpha $, then $H_{\alpha }(x)\in RV_{\theta }$
iff $x^{\alpha }\overline{F}(x)\big/H_{\alpha }(x)\sim \theta \big/(\alpha
-\theta )$.

	\item[(iii)] Let $L(x)\in RV_{0}$. Then $H_{\alpha }(x)\in \Pi _{\lambda }(L)$ iff $%
x^{\alpha }\overline{F}(x)\big/L(x)\sim \delta $ 
with $\lambda =\alpha \delta $.

\end{itemize}
\end{lemma}

\begin{proof}
(i) First, suppose that $H_{\alpha }(x)\in RV_{\theta }$ where $0\leq\theta
<\alpha $. In this case we have $x^{-\alpha -1}H(x)\in RV_{\theta -\alpha
-1} $ and, since $\theta -\alpha -1<-1$, (\ref{k2}) gives that%
\begin{equation*}
\frac{x^{-\alpha }H_{\alpha }(x)}{\int_{x}^{\infty }y^{-\alpha -1}H_{\alpha}
(y)dy}\sim \alpha -\theta \textrm{.}
\end{equation*}%
On the other hand, we have (\ref{5}): 
\begin{equation}\label{eq20210716}
\overline{F}(x)=\alpha \int_{x}^{\infty }y^{-\alpha -1}H_{\alpha
}(y)dy-x^{-\alpha }H_{\alpha }(x)\textrm{.} 
\end{equation}%
It follows that%
\begin{equation*}
\frac{\overline{F}(x)}{x^{-\alpha }H_{\alpha }(x)}\sim \frac{\alpha }{%
\alpha -\theta }-1=\frac{\theta }{\alpha -\theta }\textrm{.} 
\end{equation*}
It follows that $\overline{F}(x)\in RV_{\theta -\alpha }$ when $\theta >0$. 


Conversely, suppose that $\overline{F}(x)\in RV_{\theta -\alpha }$ and $%
0<\theta <\alpha $. Using (\ref{3}), we have%
\begin{equation*}
H_{\alpha }(x)=\alpha \int_{0}^{x}y^{\alpha -1}\overline{F}(y)dy-x^{\alpha }%
\overline{F}(x) \textrm{,}
\end{equation*}%
and, then, we find that%
\begin{equation*}
\frac{H_{\alpha }(x)}{x^{\alpha }\overline{F}(x)}\sim \frac{\alpha }{%
\theta }-1=\frac{\alpha -\theta }{\theta }\textrm{.} 
\end{equation*}%
Hence, $H_{\alpha }(x)\in RV_{\theta }$.

(ii) If $\theta \geq 0$, in (i) we proved that $H_{\alpha }(x)\in RV_{\theta
}$ implies that $x^{\alpha }\overline{F}(x)\big/H_{\alpha }(x)\sim \theta
\big/(\alpha -\theta )$. Now, we consider the converse. 
From (\ref{eq20210716}), we see that%
\begin{equation*}
\frac{\alpha \int_{x}^{\infty }z^{-\alpha -1}H_{\alpha }(z)dz}{x^{-\alpha
}H_{\alpha }(x)}\sim \frac{\theta }{\alpha -\theta }+1=\frac{\alpha }{%
\alpha -\theta }\textrm{.} 
\end{equation*}
Noting that $\theta-\alpha-1<-1$, we have, by applying (\ref{k2}), $x^{-\alpha -1}H_{\alpha }(x)\in RV_{\theta-\alpha-1}$.
We conclude that $H_{\alpha
}(x)\in RV_{\theta }$.

(iii) First, suppose that $x^{\alpha }\overline{F}(x)\big/L(x)\sim \delta $%
. Using (\ref{3}), we have%
\begin{eqnarray*}
H_{\alpha }(tx)-H_{\alpha }(x) &=&\alpha \int_{x}^{xt}y^{\alpha -1}\overline{%
F}(y)dy-(xt)^{\alpha }\overline{F}(xt)+x^{\alpha }\overline{F}(x) \\
&=&\alpha \int_{1}^{t}(xy)^{\alpha }\overline{F}(xy)y^{-1}dy-(xt)^{\alpha }%
\overline{F}(xt)+x^{\alpha }\overline{F}(x)\textrm{.}
\end{eqnarray*}%
Then, it follows that, by the dominated convergence theorem,%
\begin{eqnarray*}
\frac{H_{\alpha }(tx)-H_{\alpha }(x)}{L(x)} &=&\alpha \int_{1}^{t}\frac{%
(xy)^{\alpha }\overline{F}(xy)}{L(xy)}\frac{L(xy)}{L(x)}y^{-1}dy 
-\frac{(xt)^{\alpha }\overline{F}(xt)}{L(xt)}\frac{L(xt)}{L(x)}+\frac{%
x^{\alpha }\overline{F}(x)}{L(x)} \\
&\rightarrow &\alpha \delta \log t\textrm{.}
\end{eqnarray*}%
Hence, we get that $H_{\alpha}(x)\in \Pi _{\lambda }(L)$ with $\lambda =\alpha \delta $.

Conversely, suppose that $H_{\alpha}(x)\in \Pi _{\lambda }(L)$. Using (\ref{5}) we have 
\begin{eqnarray*}
x^{\alpha }\overline{F}(x) &=&\alpha \int_{1}^{\infty }y^{-\alpha
-1}H_{\alpha }(xy)dy-H_{\alpha }(x) \\
&=&\alpha \int_{1}^{\infty }y^{-\alpha -1}(H_{\alpha }(xy)-H_{\alpha }(x))dy\textrm{.}
\end{eqnarray*}%
Since $H_{\alpha}(x)\in \Pi _{\lambda }(L)$, we find that (cf. 
\cite{geluk}, Cor. 1.15)%
\begin{eqnarray*}
\frac{x^{\alpha }\overline{F}(x)}{L(x)} &=&\alpha \int_{1}^{\infty
}y^{-\alpha -1}\frac{H_{\alpha }(xy)-H_{\alpha }(x)}{L(x)}dy \\
&\sim &\lambda \alpha \int_{1}^{\infty }y^{-\alpha -1}\log ydy=\frac{%
\lambda }{\alpha }\textrm{.}
\end{eqnarray*}%
We conclude that $x^{\alpha }\overline{F}(x)\big/L(x)\sim \lambda \big/\alpha 
$. This proves the result.
\end{proof}


\begin{remark}\label{remarks2}
%
If $H_{\alpha }(\infty )<\infty $, then 
\begin{equation*}
m(\alpha )-H_{\alpha }(x)=\alpha \int_{x}^{\infty }y^{\alpha -1}\overline{F}%
(y)dy+x^{\alpha }\overline{F}(x)\textrm{.} 
\end{equation*}%
Using (\ref{5}), we also have%
\begin{equation*}
\overline{F}(x)=x^{-\alpha }(m(\alpha )-H_{\alpha }(x))-\alpha
\int_{x}^{\infty }z^{-\alpha -1}(m(\alpha )-H_{\alpha }(z))dz\textrm{.} 
\end{equation*}

\end{remark}



\begin{lemma}\label{lemma2bis}
Let $\alpha>0$.
Suppose that $m(\alpha )<\infty$ and $\theta >\alpha $. 
The following are equivalent:
\begin{itemize}
	\item[(i)] $\overline{F}(x)\in RV_{-\theta }$;
	\item[(ii)] $m(\alpha )-H_{\alpha }(x)\in RV_{\alpha -\theta }$; and,
	\item[(iii)] $(m(\alpha )-H_{\alpha }(x))\big/(x^{\alpha }\overline{F}(x))\sim 
\theta\big/(\theta -\alpha )$.
\end{itemize}
\end{lemma}

\begin{proof}
Assume that $\theta >\alpha $.

Proof of (i) $\Longrightarrow$ (ii).
Assume that $\overline{F}(x)\in RV_{-\theta }$.
This gives $x^{\alpha-1 }\overline{F}(x)\in RV_{\alpha-1-\theta }$.
Note that, by integration by parts, we have
\begin{equation}\label{eq20210716x}
m(\alpha )-H_{\alpha }(x)=\int_x^{\infty}y^{\alpha}dF(y)=x^{\alpha }\overline{F}(x)+\alpha\int_x^{\infty}y^{\alpha-1}\overline{F}(y)dy\textrm{.}
\end{equation}
Then, dividing (\ref{eq20210716x}) by $x^{\alpha }\overline{F}(x)$ and noting that $\alpha-1-\theta<-1 $,
we have, by applying (\ref{k2}),
$$
\frac{m(\alpha )-H_{\alpha }(x)}{x^{\alpha }\overline{F}(x)}\sim 1-\frac{\alpha}{\alpha-\theta}=\frac{\theta}{\theta-\alpha}\textrm{.}
$$
This relation implies that $m(\alpha )-H_{\alpha }(x)\in RV_{\alpha -\theta }$.
This relation also proves (i) $\Longrightarrow$ (iii).

Proof of (ii) $\Longrightarrow$ (i).
Now, assume $m(\alpha )-H_{\alpha }(x)\in RV_{\alpha -\theta }$.
Hence, $x^{-\alpha-1}(m(\alpha )-H_{\alpha }(x))\in RV_{-\theta-1 }$.
By 
Remark \ref{remarks2}, we have, noting that $-\theta-1<-1$, by applying (\ref{k2}),
$$
\frac{x^{\alpha}\overline{F}(x)}{m(\alpha )-H_{\alpha }(x)}=1-\alpha\frac{\int_x^{\infty}y^{-\alpha-1}(m(\alpha )-H_{\alpha }(y))dy}
{x^{-\alpha}(m(\alpha )-H_{\alpha }(x))}
\sim 1-\frac{\alpha}{\theta}
=\frac{\theta-\alpha}{\theta}
\textrm{.}
$$
This result implies that $x^{\alpha}\overline{F}(x)\in RV_{\alpha -\theta}$.
Hence, we get $\overline{F}(x)\in RV_{-\theta}$.

Proof of (iii) $\Longrightarrow$ (i).
Now, assume that $(m(\alpha )-H_{\alpha }(x))\big/(x^{\alpha }\overline{F}(x))\sim 
\theta\big/(\theta -\alpha )$.
We proved above (\ref{eq20210716x}).
Dividing such a relation by $x^{\alpha }\overline{F}(x)$ and applying the assumption give
%
$$
\frac{\int_x^{\infty}y^{\alpha-1}\overline{F}(y)dy}{x^{\alpha }\overline{F}(x)}\sim\frac{1}{\alpha}\left(\frac{\theta}{\theta -\alpha}-1\right)
=\frac{1}{\theta -\alpha}\textrm{.}
$$
This result and the fact that $-\theta +\alpha-1<-1$ produce, by applying (\ref{k2}), $x^{\alpha-1}\overline{F}(x)\in RV_{-\theta +\alpha-1}$.
This output leads to $\overline{F}(x)\in RV_{-\theta}$.
\end{proof}


\subsubsection{Relations between $H_{\protect\alpha }(x)$ and $\overline{G}%
_{\alpha}(x)$}


\begin{lemma}\label{lemma3}
Let $\alpha>0$ and suppose that $0\leq \theta <\alpha $. We have
$H_{\alpha }(x)\in RV_{\theta
}$ iff $\overline{G}_{\alpha}(x)\in RV_{\theta -\alpha }$ iff $x^{\alpha }\overline{G}_{\alpha}(x)\big/H_{\alpha
}(x)\sim \alpha \big/(\alpha -\theta )$.
%
%
\end{lemma}

\begin{proof}
First, suppose that $H_{\alpha}(x)\in RV_{\theta }$, $0\leq \theta <\alpha $. Recall Proposition \ref{propg} (i), i.e.
$x^{\alpha }\overline{G}_{\alpha}(x)=x^{\alpha }\overline{F}(x)+H_{\alpha
}(x)$.
Using Lemma \ref{lemma2} (ii), we find that%
\begin{equation*}
\frac{x^{\alpha }\overline{G}_{\alpha}(x)}{H_{\alpha }(x)}=\frac{x^{\alpha }%
\overline{F}(x)}{H_{\alpha }(x)}+1\sim \frac{\theta }{\alpha -\theta }%
+1=\frac{\alpha }{\alpha -\theta }\textrm{.} 
\end{equation*}
This fact proves that $H_{\alpha}(x)\in RV_{\theta }$ implies $\overline{G}_{\alpha}(x)\in RV_{\theta -\alpha }$.

Now, suppose that $\overline{G}_{\alpha}(x)\in RV_{\theta -\alpha }$. Using (\ref{5}) gives 
\begin{equation*}
\int_{x}^{\infty }y^{-\alpha -1}H_{\alpha }(y)dy=\frac{H_{\alpha
}(x)x^{-\alpha }}{\alpha }+\frac{1}{\alpha }\overline{F}(x)\textrm{.} 
\end{equation*}%
Then, we have, by Proposition \ref{propg} (ii), $\overline{G}_{\alpha}(x)=\alpha \int_{x}^{\infty }y^{-\alpha -1}H_{\alpha}(y)dy$.
A monotone density argument shows that, by using Lemma \ref{lemma2} (i),
\begin{equation*}
\frac{H_{\alpha }(x)x^{-\alpha }}{\overline{G}_{\alpha}(x)}\sim \frac{%
\alpha -\theta }{\alpha } \textrm{.}
\end{equation*}%
We conclude that $\overline{G}_{\alpha}(x)\in RV_{\theta -\alpha }$ implies that $%
H_{\alpha }(x)\in RV_{\theta }$.

Now, assume that $x^{\alpha }\overline{%
G}_{\alpha}(x)\big/H_{\alpha }(x)\sim \alpha \big/(\alpha -\theta )$.
Using again Proposition \ref{propg} (i), i.e. $x^{\alpha }\overline{G}%
_{\alpha}(x)=x^{\alpha }\overline{F}(x)+H_{\alpha }(x)$, we have%
\begin{equation*}
\frac{x^{\alpha }\overline{F}(x)}{H_{\alpha }(x)}\sim \frac{\alpha }{%
\alpha -\theta }-1=\frac{\theta }{\alpha -\theta }\textrm{.} 
\end{equation*}%
Hence, we use Lemma \ref{lemma2} (ii) for concluding tha $H_{\alpha }(x)\in RV_{\theta
}$.
\end{proof}

\subsubsection{Proofs of Theorems \ref{th1}, \ref{th2}, and \ref{th3}}

\begin{proof}[Proof of Theorem \ref{th1}]
If $0<\theta <\alpha $, then:

\begin{itemize}
	\item The relations (a) $\Longleftrightarrow$ (b) $\Longleftrightarrow$ (c) are proved by Lemma \ref{lemma1} (i).
Lemma \ref{lemma1} (i) also proves that (b) $\Longrightarrow$ (d).
The converse of the previous implication is proved by Lemma \ref{lemma1} (iii), taking $\lambda=\theta\big/\alpha$.

\item
The relation (a) $\Longleftrightarrow$ (e) is proved by Lemma \ref{lemma2} (i), whereas
(e) $\Longleftrightarrow$ (f) by Lemma \ref{lemma2} (ii).

\item
The relations (b) $\Longleftrightarrow$ (e) $\Longleftrightarrow$ (g) are proved by Lemma \ref{lemma3}.

\item
The relation (d) $\Longleftrightarrow$ (h) holds since $\overline{G}_{\alpha}(x)=\alpha x^{-\alpha}W_{\alpha}(x)$ because of Proposition \ref{propg} (ii).
%

\end{itemize}

If $\theta=0$, then:

\begin{itemize}
	\item The relations (b) $\Longleftrightarrow$ (c) and (b) $\Longrightarrow$ (d) are proved by Lemma \ref{lemma1} (ii).
The converse of (b) $\Longrightarrow$ (d) is proved by Lemma \ref{lemma1} (iii), taking $\lambda=0$.

\item
The relation (e) $\Longleftrightarrow$ (f) is proved by Lemma \ref{lemma2} (ii).

\item
The relations (b) $\Longleftrightarrow$ (e) $\Longleftrightarrow$ (g) are proved by Lemma \ref{lemma3}.

\item
The relation (d) $\Longleftrightarrow$ (h) holds since $\overline{G}_{\alpha}(x)=\alpha x^{-\alpha}W_{\alpha}(x)$ because of Proposition \ref{propg} (ii).

\end{itemize}

If $\theta=\alpha$, then the relations (a) $\Longleftrightarrow$ (b) $\Longleftrightarrow$ (c) are proved by Lemma \ref{lemma1} (i),
and also (b) $\Longrightarrow$ (d).
\end{proof}

%

\begin{proof}[Proof of Theorem \ref{th2}]
The relations (a) $\Longleftrightarrow$ (b) $\Longleftrightarrow$ (c) $\Longleftrightarrow$ (d) $\Longleftrightarrow$ (e) are proved by Lemma \ref{lemma1bis}.
The relations (a) $\Longleftrightarrow$ (f) $\Longleftrightarrow$ (g) are proved by Lemma \ref{lemma2bis}.
\end{proof}

\begin{proof}[Proof of Theorem \ref{th3}]
The relation (a) $\Longleftrightarrow$ (b) is proved by Lemma \ref{lemma1} (iv).
The relation (b) $\Longleftrightarrow$ (c) is proved by Lemma \ref{lemma2} (iii).
\end{proof}

\subsubsection{Some complements}

We briefly discuss statements about the derivatives $G_{\alpha}^{\prime }(x)$ and $%
G_{\alpha}^{\prime \prime }(x)$. To this end, we consider the inversion formula given in Proposition \ref{Newprop1}(ii), i.e.
\begin{equation}\label{eq20210716Xu}
F(x)=G_{\alpha}(x)+\frac{x}{\alpha }G_{\alpha}^{\prime }(x)\textrm{.} 
\end{equation}%
%

Suppose that $H_{\alpha }(x)\in RV_{\theta }$, $0\leq \theta <\alpha $.
By Theorem \ref{th1} (i) and (ii),
we proved that 
\begin{equation*}
\frac{x^{\alpha }\overline{G}_{\alpha}(x)}{H_{\alpha }(x)}\sim \frac{%
\alpha }{\alpha -\theta }\qquad\textrm{and}\qquad\frac{x^{\alpha }\overline{F}(x)}{H_{\alpha
}(x)}\sim \frac{\theta }{\alpha -\theta } \textrm{,}
\end{equation*}%
so it follows, by using (\ref{eq20210716Xu}), that%
\begin{equation*}
\frac{x^{1+\alpha }G_{\alpha}^{\prime }(x)}{H_{\alpha }(x)}\sim \alpha 
\textrm{.} 
\end{equation*}%
Hence, we have $G_{\alpha}^{\prime }(x)\in RV_{\theta -\alpha -1}$.

If $F$ has a probability density function (p.d.f.) $f$, we also have%
%
\begin{equation*}
x^{2+\alpha }G_{\alpha}^{\prime \prime }(x)=\alpha x^{1+\alpha }f(x)-(\alpha
+1)x^{1+\alpha }G_{\alpha}^{\prime }(x)\textrm{.} 
\end{equation*}
Moreover, if $xf(x)\big/\overline{F}(x)\sim \alpha -\theta $, we have $x^{1+\alpha
}f(x)\sim (\alpha -\theta )x^{\alpha }\overline{F}(x)\sim \theta H_{\alpha
}(x)$, and it follows that%
\begin{equation*}
\frac{x^{2+\alpha }G_{\alpha}^{\prime \prime }(x)}{H_{\alpha }(x)}=\frac{%
x^{1+\alpha }f(x)}{H_{\alpha }(x)}-\frac{\alpha +1}{\alpha }\frac{%
x^{1+\alpha }G_{\alpha}^{\prime }(x)}{H_{\alpha }(x)}\sim \theta -\alpha -1%
\textrm{.} 
\end{equation*}

In the special case that $m(\alpha )<\infty $, we have the following
corollary.

\begin{corollary}
Let $\alpha>0$.
Suppose that $m(\alpha )<\infty $. 
We have $x^{\alpha }\overline{F}%
(x)\sim 0$, $x^{\alpha }\overline{G}_{\alpha}(x)\sim m(\alpha )$
and $x^{1+\alpha }G_{\alpha}^{\prime }(x)\sim \alpha m(\alpha )$. If $F$ has a p.d.f. $f(x)$, and also, for some $\delta\in\mathbb{R}$, $%
x^{1+\delta }f(x)=o(1)$ or $xf(x)=O(\overline{F}(x))$, then $x^{2+\delta
}G_{\alpha}^{\prime \prime }(x)\sim -\alpha (\alpha +1)m(\alpha )$.
\end{corollary}





\end{document}